\documentclass[pdflatex,sn-mathphys-num]{sn-jnl}
\usepackage{amsmath,amssymb,mathtools}
\usepackage{xcolor}
\usepackage{pgfplots}
\usepgfplotslibrary{groupplots}
\pgfplotsset{compat=1.18}
\definecolor{deepblue}{HTML}{2457A6}
\definecolor{warmred}{HTML}{C74440}
\hypersetup{colorlinks=true,linkcolor=deepblue,citecolor=deepblue,urlcolor=deepblue,
  pdftitle={A Note on Sphere Packing Bounds for Tuple Lattice Sieving},pdfauthor={Thijs Laarhoven},
  pdfsubject={Upper and lower bounds for tuple-irreducible spherical codes}}
\theoremstyle{thmstyleone}
\newtheorem{theorem}{Theorem}[section]
\newtheorem{lemma}[theorem]{Lemma}
\newtheorem{corollary}[theorem]{Corollary}
\theoremstyle{thmstyletwo}
\newtheorem{remark}[theorem]{Remark}
\newtheorem{definition}[theorem]{Definition}
\newcommand{\Rrate}{\mathcal{R}}
\newcommand{\Urate}{\mathcal{U}}
\newcommand{\Lrate}{\mathcal{L}}
\newcommand{\vx}{\mathbf{x}}
\newcommand{\vy}{\mathbf{y}}
\newcommand{\vz}{\mathbf{z}}
\newcommand{\vu}{\mathbf{u}}
\newcommand{\vs}{\mathbf{s}}
\newcommand{\sph}{\mathbb S}
\newcommand{\RR}{\mathbb R}
\newcommand{\EE}{\mathbb E}
\newcommand{\norm}[1]{\lVert#1\rVert}
\newcommand{\ev}{\mathrm{ev}}
\newcommand{\KL}{\mathrm{KL}}
\newcommand{\HK}{\mathrm{HK}}
\newcommand{\OAI}{\mathrm{OAI}}
\begin{document}

\title[Packing bounds for tuple lattice sieving]{A Note on Sphere Packing Bounds for Tuple Lattice Sieving}
\author*{\fnm{Thijs} \sur{Laarhoven}}\email{mail@thijs.com}

\abstract{A finite set of unit vectors is $k$-irreducible if every signed sum of between two and $k$ distinct elements has norm greater than one. Let $\Rrate_k$ be the maximal asymptotic rate of such sets, and let $\kappa(\alpha)$ be the maximal asymptotic rate of spherical codes with pairwise inner products at most $\alpha$. For $k \ge 2$ we show:
\begin{align}
\Rrate_k \le \min_{1 \le r \le \lfloor k/2 \rfloor} \frac{1}{r} \, \kappa\!\left(1 - \frac{1}{2r}\right) \, .
\end{align}
Combining this with standard sphere packing bounds, for large $k$ we obtain an almost-tight asymptotic comparison with the known lower bounds:
\begin{align}
\left(\tfrac{1}{2}-o(1)\right) \, \frac{\log_2 k}{k} \le \Rrate_k \le (1 + o(1)) \, \frac{\log_2 k}{k} \, .
\end{align}
}

\maketitle

%%%%%%%%%%%%%%%%%%%%%%%%%%%%%%%%%%%%%%%%%%%%%%%%%%%%%%%%%%%%%%%%%%%%%
%%%%%%%%%%%%%%%%%%%%%%%%%%%%%%%%%%%%%%%%%%%%%%%%%%%%%%%%%%%%%%%%%%%%%
%%%%%%%%%%%%%%%%%%%%%%%%%%%%%%%%%%%%%%%%%%%%%%%%%%%%%%%%%%%%%%%%%%%%%

\section{Introduction}
\label{sec:introduction}

Lattice sieving seeks short lattice vectors by repeatedly combining elements of a list. The classical pairwise operation replaces a vector using a shorter sum or difference. For an equal-length list, the absence of such reductions imposes an angular separation: after scaling to unit length, distinct vectors satisfy $|\langle \vx_i,\vx_j\rangle|<1/2$ if sums of norm at most one are excluded. The list is therefore a spherical code, and kissing-number upper bounds control its size. This geometric observation is central to the analysis of GaussSieve-type lists~\cite{MV10}.

Bai, Laarhoven, and Stehl\'e (BLS) introduced tuple lattice sieving to reduce the memory requirement by permitting combinations of more than two vectors~\cite{BLS16}. Their analysis predicted a decreasing list-size exponent, subsequently established for independent random spherical inputs by Herold and Kirshanova (HK)~\cite{HK17}. Subsequent work involved further classical~\cite{HKL18} and quantum improvements to tuple lattice sieving~\cite{KMPR19, ECG26}, illustrating continuing interest in these methods. These algorithmic developments leave a distinct extremal question: how large can a spherical set be if \emph{none} of its permitted signed combinations is short?

%%%%%%%%%%%%%%%%%%%%%%%%%%%%%%%%%%%%%%%%%%%%%%%%%%%%%%%%%%%%%%%%%%%%%
%%%%%%%%%%%%%%%%%%%%%%%%%%%%%%%%%%%%%%%%%%%%%%%%%%%%%%%%%%%%%%%%%%%%%
%%%%%%%%%%%%%%%%%%%%%%%%%%%%%%%%%%%%%%%%%%%%%%%%%%%%%%%%%%%%%%%%%%%%%

\section{Spherical codes}
\label{sec:spherical}

All logarithms are to base two unless written as $\ln$. The dimension is $n\ge2$, and $k$ is fixed whenever $n\to\infty$. We write $\sph^{n-1}=\{\vx\in\RR^n:\norm{\vx}=1\}$, with the Euclidean norm. All lists in this note are finite sets; indices in a tuple are distinct.

\begin{samepage}
\begin{definition}[Spherical codes and their rates]
\label{def:codes}
For $-1\le\alpha<1$, let $A(n,\alpha)$ be the largest cardinality of a set $C\subset\sph^{n-1}$ satisfying $\langle\vx,\vy\rangle\le\alpha$ for all distinct $\vx,\vy\in C$. Its asymptotic rate is
\begin{align}\label{eq:kappa}
 \kappa(\alpha)=\limsup_{n\to\infty}\frac1n\log_2 A(n,\alpha).
\end{align}
\end{definition}
\end{samepage}

The parameter $\alpha$ is the maximum inner product. The equivalent minimum Euclidean distance is $\sqrt{2(1-\alpha)}$. In particular, $A(n,1/2)$ is the ordinary kissing number in dimension $n$, and $\kappa(1/2)$ is its asymptotic exponent.

\begin{definition}[Tuple irreducibility]
\label{def:tuple}
A set $C=\{\vx_1,\ldots,\vx_N\}\subset\sph^{n-1}$ is \emph{irreducible
through order $k$} if
\begin{align}\label{eq:irreducible}
 \left\|\sum_{i\in I}s_i \vx_i\right\|>1
 \quad\text{for all }2\le |I|\le k
 \quad\text{and all }s_i\in\{-1,1\}.
\end{align}
Let $N_k(n)$ be the largest size of such a set, and put
$\Rrate_k=\limsup_{n\to\infty}n^{-1}\log_2 N_k(n)$.
For $k=2m$, define $N^{\ev}_{2m}(n)$ and $\Rrate^{\ev}_{2m}$ analogously,
but require~\eqref{eq:irreducible} only for $|I|\in\{2,4,\ldots,2m\}$.
\end{definition}

The permitted nonzero coefficients are exactly $\pm1$; repetitions and coefficients of larger magnitude are not part of the definition. Zero sums are excluded as well. The condition is cumulative: irreducibility at the largest support size alone is insufficient for our upper-bound argument. We use the strict convention in~\eqref{eq:irreducible} throughout. For random spherical tuples the boundary has probability zero. In particular,
\begin{align}\label{eq:monotonicity}
 N_{k+1}(n)\le N_k(n),\qquad
 N_{2m}(n)\le N^{\ev}_{2m}(n)\le A(n,1/2).
\end{align}
At $k=2$ the signed condition is $|\langle\vx,\vy\rangle|<1/2$. Thus $\Rrate_2\le\kappa(1/2)$; we do not identify this restricted signed-code rate with the ordinary kissing-number rate.

We use the following numerical forms of the classical KL bound and the recent OpenAI bound. The constants are rounded outwards.

\begin{lemma}[Kabatiansky--Levenshtein bound~\cite{KL78}]\label{lem:kl}
For $1/2\le\alpha<1$,
\begin{align}\label{eq:kl}
 \kappa(\alpha)\le\kappa^{(\KL)}(\alpha)
 :=-\frac12\log_2\bigl(2(1-\alpha)\bigr)+0.400944234.
\end{align}
\end{lemma}

\begin{lemma}[OpenAI bound~\cite{OAI26}]\label{lem:oai}
For $1/2\le\alpha<1$,
\begin{align}\label{eq:oai}
 \kappa(\alpha)\le\kappa^{(\OAI)}(\alpha)
 :=-\frac12\log_2\bigl(2(1-\alpha)\bigr)+
 \begin{cases}
  0.396602,&1/2\le\alpha<3/4,\\
  0.395615,&3/4\le\alpha<1.
 \end{cases}
\end{align}
Moreover, the full hierarchy gives $\kappa(\alpha)\le\tfrac12\log_2(e/(\pi(1-\alpha)))+o(1)$ as $\alpha\uparrow1$.
\end{lemma}

Lemma~\ref{lem:oai} uses Chapter~2, Theorem~1.2 and its spherical-cap form~(12), and Theorem~8.3 of~\cite{OAI26}. The displayed decimals come from explicit level-two parameters supplied with the numerical code.

We finally record the rate lower bound used for comparison.

\begin{lemma}[Rate lower bound from Herold--Kirshanova]\label{lem:hk-lower}
For $k\ge2$, with the expansion as $k\to\infty$,
\begin{align}\label{eq:lower-rate}
 \Rrate_k\ge\Lrate_k^{(\HK)}
 :=\frac12\left(\frac{\log_2k}{k-1}-\log_2\left(1+\tfrac{1}{k}\right)\right)
 =\frac{\log_2k-\log_2e}{2k}+O\!\left(\frac{\log k}{k^2}\right).
\end{align}
Moreover, $\Lrate_k^{(\HK)}\ge(\log_2k-\log_2e)/(2k)$ for every $k\ge2$.
\end{lemma}

The rate bound is the standard alteration consequence of the configuration estimates in \cite[Corollary~1, based on Theorems~1--2]{HK17}, applied to all signed support sizes from two through $k$. The exponent was predicted in \cite[Eq.~(3.2)]{BLS16} and established as the random-list exponent in \cite[Theorem~3]{HK17}. The expansion and the last inequality are elementary consequences of the formula for $\Lrate_k^{(\HK)}$.

%%%%%%%%%%%%%%%%%%%%%%%%%%%%%%%%%%%%%%%%%%%%%%%%%%%%%%%%%%%%%%%%%%%%%
%%%%%%%%%%%%%%%%%%%%%%%%%%%%%%%%%%%%%%%%%%%%%%%%%%%%%%%%%%%%%%%%%%%%%
%%%%%%%%%%%%%%%%%%%%%%%%%%%%%%%%%%%%%%%%%%%%%%%%%%%%%%%%%%%%%%%%%%%%%

\section{The upper bound}
\label{sec:upper}

\begin{theorem}[Half-tuple packing bound]
\label{thm:upper}
% For $k\ge2$ and $n \geq 2$,
% \begin{align}\label{eq:upper-size-function}
%  N_k(n)\le\min_{1\le r\le\lfloor k/2\rfloor}
%  \left\{r-1+\left[
%   2^{r-1}r!\,A\!\left(n+1,1-\frac1{2r}\right)
%  \right]^{1/r}\right\}.
% \end{align}
For every fixed $k\ge2$,
\begin{align}\label{eq:main-upper}
 \Rrate_k \le \min_{1\le r\le\lfloor k/2\rfloor} \frac{1}{r} \, \kappa\!\left(1-\frac1{2r}\right).
\end{align}
\end{theorem}

The following three lemmas prepare the proof of Theorem~\ref{thm:upper}. A common signing yields many short half-tuple sums, irreducibility separates them, and a lifting turns them into a spherical code. 

\begin{lemma}[A global signing with many short sums]\label{lem:signing}
Let $\vx_1,\ldots,\vx_N\in\sph^{n-1}$ and $1\le r\le N$. There is a signing $\vs=(s_1,\ldots,s_N)\in\{-1,1\}^N$ such that, writing $\vy_A=\sum_{i\in A}s_i\vx_i$,
\begin{align}\label{eq:short-sums}
 \#\left\{A\in\binom{[N]}r:\norm{\vy_A}^2\le r\right\}
 \ge2^{1-r}\binom Nr.
\end{align}
\end{lemma}

\begin{proof}
Choose the entries of $\vs$ independently and uniformly. For every fixed $r$-subset $A$,
\begin{align}
 \EE_{\vs}\norm{\vy_A}^2
 =\sum_{i\in A}\norm{\vx_i}^2+
   \sum_{\substack{i,j\in A, i \neq j}}
   \EE(s_i s_j)\langle\vx_i,\vx_j\rangle=r.
\end{align}
At least one of the $2^r$ sign patterns on $A$ therefore has squared norm at most $r$. Its negative has the same norm, so at least two patterns do. Consequently, the probability that $\norm{\vy_A}^2\le r$ is at least $2^{1-r}$. By linearity of expectation, the expected number of such subsets $A$ is at least $2^{1-r}\binom Nr$. Some global signing attains at least this number, proving~\eqref{eq:short-sums}.
\end{proof}

\begin{lemma}[Separation of half-tuple sums]
\label{lem:separation}
Suppose $C$ is even-irreducible through order $2m$. For any global signing $\vs$ and any $r\le m$, the sums $\vy_A$ over distinct $r$-subsets obey $\norm{\vy_A-\vy_B}>1$. In particular, all these sums are distinct.
\end{lemma}

\begin{proof}
For $A\ne B$, cancellation gives
\begin{align}
\vy_A-\vy_B=\sum_{i\in A\setminus B}s_i \vx_i-\sum_{i\in B\setminus A}s_i \vx_i.
\end{align}
Because $|A|=|B|=r$, the support $A\triangle B$ has size $2(r-|A\cap B|)\in\{2,4,\ldots,2r\}$. All its coefficients are signs, so even irreducibility proves the claim. This also explains why forbidding only support size $2m$ would not suffice: overlapping subsets can leave smaller supports.
\end{proof}

\begin{lemma}[Lifting a ball packing to a sphere]
\label{lem:lift}
Let $Y\subset\RR^n$ lie in the closed ball of radius $\rho>0$ centered at the origin, and suppose distinct points of $Y$ have distance greater than $\ell>0$. If $\ell/\rho\le2$, then $|Y|\le A(n+1,1-\ell^2/(2\rho^2))$.
\end{lemma}

\begin{proof}
Map each $\vy\in Y$ to $\vu_{\vy}=\rho^{-1}\bigl(\vy,\sqrt{\rho^2-\norm{\vy}^2}\bigr)\in\sph^n$. The first $n$ coordinates make the map injective, and
\begin{align}
 \norm{\vu_{\vy}-\vu_{\vz}}^2
 =\frac{\norm{\vy-\vz}^2+(\sqrt{\rho^2-\norm{\vy}^2}-\sqrt{\rho^2-\norm{\vz}^2})^2}{\rho^2}
 >\frac{\ell^2}{\rho^2}.
\end{align}
Since the images are unit vectors, $\langle\vu_{\vy},\vu_{\vz}\rangle=1-\norm{\vu_{\vy}-\vu_{\vz}}^2/2 <1-\ell^2/(2\rho^2)$. Thus they form the required spherical code.
\end{proof}

\begin{proof}[Proof of Theorem~\ref{thm:upper}]
Put $m=\lfloor k/2\rfloor$. Fix $r\le m$ and a set of size $N\ge r$ that is even-irreducible through order $2m$. Lemmas~\ref{lem:signing} and~\ref{lem:separation} produce at least $2^{1-r}\binom Nr$ distinct, mutually separated points in a ball of radius $\sqrt r$. Lemma~\ref{lem:lift}, with $\rho=\sqrt r$ and $\ell=1$, gives
\begin{align}\label{eq:finite-upper}
 2^{1-r}\binom Nr\le A\!\left(n+1,1-\frac1{2r}\right).
\end{align}
Minimizing over $r$ proves the cardinality bound for even-irreducible sets. Full irreducibility implies these even conditions, so the bound applies to $N_k(n)$ as well.

For each fixed $r$, the same estimate gives
\begin{align}
 \log_2N\le\frac1r\log_2 A\!\left(n+1,1-\frac1{2r}\right)+O_r(1).
\end{align}
Divide by $n$, take the upper limit as $n\to\infty$, and use $(n+1)/n\to1$. The definition of $\kappa$ then gives $\Rrate^{\ev}_{2m}\le r^{-1}\kappa(1-1/(2r))$. Minimizing over $r\le m$ proves the rate statements, including for odd $k$, since $\Rrate_k\le\Rrate^{\ev}_{2m}$.
\end{proof}

The factor $1/r$ comes from converting roughly $N^r$ half-tuple sums into a spherical code. Their radius is $\sqrt r$, so lifting and scaling give pairwise inner products below $1-1/(2r)$. These are the two geometric features responsible for the large-$k$ behavior.

Substituting $\kappa^{(\OAI)}$ from Lemma~\ref{lem:oai} into Theorem~\ref{thm:upper}, with $r=1$ and $r=k/2$, gives the following explicit form of our upper bound.

\begin{corollary}[Explicit upper bound]
\label{cor:explicit}
For even $k\ge4$,
\begin{align}\label{eq:rate-upper}
 \Rrate_k\le\Urate_k
 :=\min\left\{0.396602,
             \frac{\log_2k-0.208770}{k}\right\}.
\end{align}
The same upper bound holds for $\Rrate^{\ev}_k$. For $k=2$, it reads $\Rrate_2\le\Urate_2:=0.396602$. For odd $k$, the corresponding bounds at $k-1$ apply.
\end{corollary}

\begin{remark}[The cases $k=2$ and $k=4$]
\label{rem:small-k}
For $k=2$, Theorem~\ref{thm:upper} gives $\Rrate_2\le\kappa(1/2)$, recovering the pairwise bound. At $k=4$, substituting $\kappa^{(\OAI)}$ gives the half-tuple term $0.4478075$, which exceeds the pairwise estimate $0.396602$. Our explicit upper bound therefore retains the pairwise estimate at $k=4$, and first improves on it at $k=6$. This does not imply that the true four-tuple rate equals the pairwise rate.
\end{remark}

%%%%%%%%%%%%%%%%%%%%%%%%%%%%%%%%%%%%%%%%%%%%%%%%%%%%%%%%%%%%%%%%%%%%%
%%%%%%%%%%%%%%%%%%%%%%%%%%%%%%%%%%%%%%%%%%%%%%%%%%%%%%%%%%%%%%%%%%%%%
%%%%%%%%%%%%%%%%%%%%%%%%%%%%%%%%%%%%%%%%%%%%%%%%%%%%%%%%%%%%%%%%%%%%%

\section{Large-\texorpdfstring{$k$}{k} asymptotics}
\label{sec:asymptotics}

We now let $k$ grow \emph{after} taking the dimension limit defining each rate. The previous estimates are not asserted to hold uniformly when $k=k(n)$ grows with $n$.

\begin{theorem}[Two-sided asymptotics]\label{thm:asymptotics}
As $k\to\infty$ through integers,
\begin{align}\label{eq:two-sided-asymptotics}
 \frac{\log_2k-\log_2e}{2k}
 \le\Rrate_k\le
 \frac{\log_2k+\log_2(e/\pi)+o(1)}{k}.
\end{align}
The same statement holds for $\Rrate^{\ev}_k$ through even $k$. In particular, $\Rrate_k=\Theta((\log k)/k)$ and
\begin{align}\label{eq:normalized-limits}
 \frac12\le\liminf_{k\to\infty}\frac{k\Rrate_k}{\log_2k}
 \le\limsup_{k\to\infty}\frac{k\Rrate_k}{\log_2k}\le1.
\end{align}
\end{theorem}
\begin{proof}
The lower bound is contained in Lemma~\ref{lem:hk-lower}. For even $k$, take $r=k/2$ in Theorem~\ref{thm:upper} and use the small-angle estimate in Lemma~\ref{lem:oai}:
\begin{align}
 \Rrate^{\ev}_k\le\frac2k\kappa(1-1/k)
 \le\frac{\log_2k+\log_2(e/\pi)+o(1)}{k}.
\end{align}
For odd $k$, apply the even estimate at $k-1$ and use $\Rrate_k\le\Rrate^{\ev}_{k-1}$. Replacing $k-1$ by $k$ costs $O((\log k)/k^2)=o(1/k)$. Dividing by $(\log_2k)/k$ proves \eqref{eq:normalized-limits}.\qedhere
\end{proof}

The upper and lower bounds in Fig.~\ref{fig:comparison} are asymptotically a factor two apart: $\Urate_k/\Lrate_k^{(\HK)}\to2$ as $k\to\infty$ through even integers.

\begin{figure}[!t]
\centering
\begin{tikzpicture}
\begin{groupplot}[
 group style={group size=2 by 1,horizontal sep=2.5cm},
 width=0.44\textwidth,height=0.20\textheight,
 xmode=log,log basis x=2,xlabel={even tuple order $k$},
 grid=major,grid style={draw=gray!18},
 tick label style={font=\scriptsize},label style={font=\footnotesize},
]
\nextgroupplot[ylabel={$\Rrate_k$},ymin=0,
 legend pos=north east,
 legend style={font=\scriptsize,draw=none,fill=none,legend columns=1},
 title={(a) Rates},title style={font=\footnotesize}]
\addplot+[deepblue,thick,mark=*,mark size=1.6pt,
 mark options={fill=white,draw=deepblue,line width=0.5pt}]
 table[x=k,y=oai]{bounds_data.dat};
\addlegendentry{$\Urate_k$}
\addplot+[warmred,thick,mark=square*,mark size=1.3pt,
 mark options={fill=warmred,draw=warmred}]
 table[x=k,y=lower]{bounds_data.dat};
\addlegendentry{$\Lrate_k^{(\HK)}$}
\nextgroupplot[ylabel={$\Rrate_k \cdot (k / \log_2 k)$},ymin=-0.1,ymax=1.1,
 ytick={0,0.25,0.5,0.75,1},
 title={(b) Normalized rates},title style={font=\footnotesize}]
\addplot[deepblue,thick,densely dotted,forget plot]
 coordinates {(2,1) (1024,1)};
\addplot[warmred,thick,densely dotted,forget plot]
 coordinates {(2,0.5) (1024,0.5)};
\addplot+[deepblue,thick,mark=*,mark size=1.6pt,
 mark options={fill=white,draw=deepblue,line width=0.5pt}]
 table[x=k,y=oai_norm]{bounds_data.dat};
\addplot+[warmred,thick,mark=square*,mark size=1.3pt,
 mark options={fill=warmred,draw=warmred}]
 table[x=k,y=lower_norm]{bounds_data.dat};
\end{groupplot}
\end{tikzpicture}
\caption{The lower bound $\Lrate_k^{(\HK)}$ from Lemma~\ref{lem:hk-lower} and our upper bound $\Urate_k$ from~\eqref{eq:rate-upper}. Left: the two rate bounds. Right: the same curves divided by $(\log_2k) / k$, with horizontal asymptotes $1/2$ and $1$.}
\label{fig:comparison}
\end{figure}
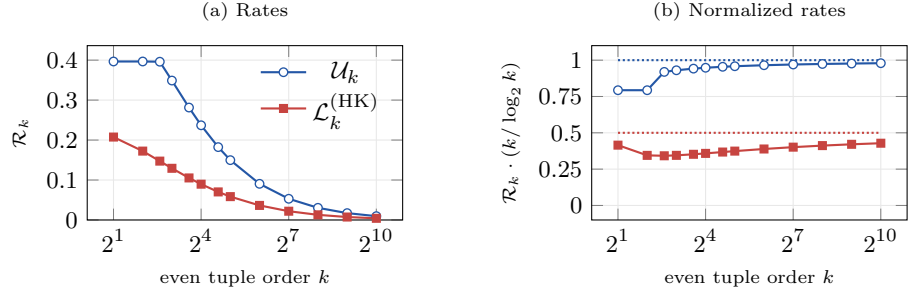

% \backmatter
% \begin{samepage}
% \section*{Statements and Declarations}
% \noindent\textbf{Funding.} The author received no specific funding for this work.

% \noindent\textbf{Competing interests.} The author has no relevant interests to disclose.

% \noindent\textbf{Data and code availability.} The source package contains numerical data for Fig.~\ref{fig:comparison}.

% \noindent\textbf{Use of generative artificial intelligence.} OpenAI language models assisted with exploration, numerical evaluation, and manuscript preparation.
% \end{samepage}

\setlength{\bibsep}{0.25\baselineskip}

\end{document}